\documentclass{commat}

\usepackage{dsfont}

\DeclareMathAlphabet{\pazocal}{OMS}{zplm}{m}{n}
\newtheorem{assumption}{Assumption}

\title{%
    Upper semicontinuity of the attractor for a nonlinear hyperbolic-parabolic coupled system with fractional Laplacian 
    }

\author{%
    Manoel dos Santos and Renato Lobato
    }

\authorinfo[%
    Manoel Dos Santos]{
    Federal University of Pará, Brazil}{%
    jeremias@ufpa.br
    }

\authorinfo[%
    Renato Lobato]{
    Federal University of Pará, Brazil}{%
    renatolobato@ufpa.br
    }

\abstract{%
    In this paper we establish the existence and uniqueness of global solutions (in time), as well as the existence, regularity and stability (upper semicontinuity) of the attractors for the semigroup generated by the solutions of a two-dimensional nonlinear hyperbolic-parabolic coupled system with fractional Laplacian. In addition, we also obtain the existence of an exponential attractor and show that this attractor has a finite fractal dimension in a space containing the phase space of the dynamical system.
    }

\keywords{%
    Hyperbolic-parabolic coupled system, Global and Exponential Attractors, Upper Semicontinuity of the attractors, Fractional Laplacian.
    }

\msc{%
    35B40, 35B41, 37L30.
    }

\VOLUME{32}
\YEAR{2024}
\NUMBER{1}
\firstpage{281}
\DOI{https://doi.org/10.46298/cm.11749}

\begin{document}

\section{Introduction}\label{sec:1}

In this work, we study the long-time behavior of the solutions of the following two-dimensional nonlinear hyperbolic-parabolic coupled  system with fractional Laplacian
\begin{eqnarray}
	\begin{aligned}
		u_{tt}+Au+A^\nu u_t-\delta A^\sigma \theta
		+f(u)= h,\quad\text{in} \ \Omega\times [0,\infty],\\
		\theta_t+ A\theta+\delta A^\sigma u_{t}= 0,\quad\, \text{in}\ \Omega\times[0,\infty],
	\end{aligned}\label{p:1:1}
\end{eqnarray}
where $\delta$ is a positive constant, $\nu,\sigma\in[0,1]$ and  $\Omega\subset\mathds{R}^2$ is a domain with smooth boundary $\partial\Omega$. We consider Dirichlet boundary conditions, that is 
\begin{eqnarray}
	u=\theta=0,\quad \text{on} \quad \partial\Omega\times[0,\infty] ,\label{p:1:2}
\end{eqnarray}
and initial conditions
\begin{eqnarray}
	\begin{aligned}
		u(\cdot,0)=u_0(\cdot),\quad  u_t(\cdot,0)=u_1(\cdot),\quad \theta(\cdot,0)=\theta_0(\cdot),& \quad\text{in}\quad  \Omega,
	\end{aligned}\label{p:1:3}
\end{eqnarray}
where $u_0$, $u_1$ and $\theta$ are known functions belonging to appropriate spaces. The operator $A^\gamma:D(A^\gamma)\subset L^2(\Omega)\to L^2(\Omega)$ is the fractional power of order $\gamma\in\mathds{R}$ associated with Laplacian operator \begin{eqnarray}
	A=-\Delta:D(A)\subset L^2(\Omega)\to L^2(\Omega),\quad\text{with domain}\quad D(A)=H^2(\Omega)\cap H_0^1(\Omega).
\end{eqnarray}

Physically, $u$ represents the transverse displacement with respect to $\Omega$ and $\theta$ is the temperature. The system \eqref{p:1:1} possesses an important characteristic: the first equation is nonlinear hyperbolic in unknown $u$ and the second equation is linear parabolic with respect to unknown $\theta$, that is, \eqref{p:1:1} is a nonlinear hyperbolic-parabolic coupled system. Systems of this nature, besides modeling the kinetics of a viscoelastic body with heat conduction (thermoviscoelastic systems), can also model the kinetics of gases and liquids \cite{isbn:9781498749640}. There are many studies considering the asymptotic behavior of thermoelastic system solutions, among them \cite{LIU-ZHENG,RIVERA-FE-1992,PEREIRA-MENZALA-CAM-1989,isbn:9783764388140,doi:10.1017/S0308210500031164,doi:10.1002/mma.1670120308,doi:10.1007/BF00375601,doi:10.2307/43637959,isbn:9781498749640}.

The solutions of the system \eqref{p:1:1}-\eqref{p:1:3} give rise to a family of nonlinear semigroups $S_\beta(t)$, indexed in $\beta=(\nu,\sigma)\in\Lambda:=[0, 1]\times[0, 1]$ defined in a phase space $\pazocal{H}$, therefore resulting in a family of dynamical systems $(\pazocal{H},S_\beta(t))$.  Then, there are issues of considerable interest such as:
\begin{description}
	\item[a)] Existence of global and exponential attractor, $\mathfrak{A}_\beta$ and $\mathfrak{A}_\beta^{exp}$, respectively for $(\pazocal{H}, S_\beta(t))$, $\beta\in\Lambda$;
	\item[b)] The finite dimensionality of the global attractor;
	\item[c)] Regularity of trajectories from the attractor $\mathfrak{A}_\beta$;
	\item[d)] non-explosion property (upper semicontinuity) of attractors family $\{\mathfrak{A}_\beta\}_{\beta\in\Lambda}$ at $(0,0)$. 
\end{description}

The structure of the dynamical system (phase space and evolution operator), obtained by solutions of \eqref{p:1:1}-\eqref{p:1:3}, allows us to apply the concept of quasi-stability, introduced by I. Chueshov and I. Lasicka in \cite{doi:10.1007/s10884-004-4289-x} (see also \cite{doi:10.1007/978-3-319-22903-4,isbn:9780821866535,ISBN:9780387-877112}) to the study of the existence of the global attractor, as well as to the study of some properties of this global attractor. In this work, we will show that the dynamical system $(\pazocal{H},S_\beta (t))$ is quasi-stable on bounded positively invariant sets, for this, we must prove a inequality known as stabilizability estimate \eqref{6:5a}.

The finite fractal dimensionality of global attractor is a very interesting property, there are several approaches for obtaining it, the fractal dimension is a concept of dimension measurement widely used in the theory of infinite dimensional dynamical systems. A compact set with a finite fractal dimension can be considered as the image of a compact subset of $\mathds{R}^n$, for some $n$, by a continuous Hölder application (see \cite{doi:10.1007/978-3-319-22903-4,doi:10.1007/s10884-004-4289-x,isbn:9780821866535,ISBN:9780387-877112,doi:10.2307/24899128}).

To investigate the upper semicontinuity, note that the parameters $\beta\in\Lambda$ gives rise to a family of semigroups $S_\beta(t)$ generated by the solutions of \eqref{p:1:1}-\eqref{p:1:3} and let us prove that the dynamical system $(\pazocal{H},S_\beta(t))$ possesses a compact global attractor $\mathfrak{A}_\beta$ for any $\beta\in\Lambda$. In addition, the dynamical system $(\pazocal{H},S_0(t))$, where $S_{\beta_0}(t)$ is the semigroup generated by the limiting solutions $\beta_0=(0,0)$, that is	
\begin{eqnarray}
	\begin{aligned}
		u_{tt}+Au+u_t-\delta \theta
		+f(u)=&\ h,\quad\text{in} \ \Omega\times[0,\infty],\\
		\theta_t+A\theta+\delta  u_{t}=&\ 0,\quad\, \text{in}\ \Omega\times[0,\infty],
	\end{aligned}\label{1:1}
\end{eqnarray}
also possesses a compact global attractor $\mathfrak{A}_{\beta_0}$. We must prove that
\begin{eqnarray}
	\lim_{\beta\to 0}d_{\pazocal{H}}\{\mathfrak{A}_\beta|\mathfrak{A}_{\beta_0}\}=0,\label{1:2}
\end{eqnarray}
where $d_{\pazocal{H}}\{\mathfrak{A}_\beta|\mathfrak{A}_{\beta_0}\}$ is the Hausdorff semidistance between $\mathfrak{A}_\beta$ and $\mathfrak{A}_{\beta_0}$ defined in Section \ref{sec:4}.

\section{Notations and Assumptions}\label{sec:2}

In this paper, we represent by $\|\cdot\|_q$ the norm in the Lebesgue space $L^q(\Omega)$ for $1\leq q\leq\infty$. When $q=2$, we user $(\cdot,\cdot)_2$ to represent the inner product in $L^2(\Omega)$ that is given by. In addition, we use standard notations for Sobolev spaces $H^i(\Omega)$, for $i=1,2,3$, and $H^{-1}(\Omega)$ is the dual of $H_0^1(\Omega)$. For $\gamma\in\mathds{R}$, we consider the Hilbert spaces $D(A^{\gamma/2})$ endowed with the inner product and the norm given, respectively, by
\begin{eqnarray}
	(u,v)_{D(A^{\gamma/2})}=(A^{\gamma/2}u, A^{\gamma/2}v)_2\quad\text{and}\quad \|u\|_{D(A^{\gamma/2})}=\|A^{\gamma/2}u\|_2.\label{3:1}
\end{eqnarray} 
Moreover, for $\gamma_2>\gamma_1$ with $0\leq \gamma_2-\gamma_1\leq 1$ the embeddings $D(A^{\gamma_2/2})\hookrightarrow D(A^{\gamma_1/2})$ are compact and
\begin{eqnarray}
	\|u\|_{D(A^{\gamma_1/2})}\leq \kappa\|u\|_{D(A^{\gamma_2/2})},\quad\forall u\in D(A^{\gamma_2/2}), \quad \text{with}\quad \kappa=\max\{\lambda_1^{-1/2},1\},\label{3:1a}
\end{eqnarray}
where $\lambda_1$ is the first eigenvalue of Laplacian (in this case, embedding constants $\kappa$ do not depend on $\gamma_1$ or $\gamma_2$). We adopt the following notations for the spaces
\begin{eqnarray}
	V_0=L^2(\Omega),\quad V_1=D(A^{1/2})=H_0^1(\Omega), \quad V_2=D(A)=H^2(\Omega)\cap H_0^1(\Omega).
\end{eqnarray}
Throughout this paper, we consider the following assumptions:
\begin{assumption}\label{asp:3:1}
	For external force, let us consider $h\in L^2(\Omega)$.
\end{assumption}

\begin{assumption}\label{asp:3:2}
	Concerning the nonlinear term $f:\mathds{R}\to\mathds{R}$, we assume that $f$ is $C^1(\mathds{R})$ function and there exist $p\geq 0$, $C_f>0$ and $m_F>0$ such that
	\begin{eqnarray}
		|f'(s)|\leq C_f(1+|s|^p), \quad\forall s\in\mathds{R}\label{3:2}
	\end{eqnarray}
	and 
	\begin{eqnarray}
		f(s)s\geq 0\quad\text{and}\quad F(s)\geq  -m_F, \quad\forall s\in\mathds{R},\label{3:3}
	\end{eqnarray}
	where $F(s)=\int_0^sf(\tau)d\tau$.
\end{assumption}
\begin{remark}
	Note that integrating \eqref{3:2} over $[0,s]$, we obtain 
	\begin{eqnarray}
		\begin{aligned}
			|f(s)-f(0)|\leq C_f\bigg(|s|+\frac{|s|^{p+1}}{p+1}\bigg)\leq C_f(|s|+|s|^{p+1})\leq \\
			C_f(|s|^{p+1}+1+|s|^{p+1})\leq C_f(1+2|s|^{p+1}).
		\end{aligned}
	\end{eqnarray}
	From \eqref{3:3} $(f(s)s\geq 0)$ we have $f(0)=0$. Therefore,
	\begin{eqnarray}
		|f(s)|\leq C_f(1+2|s|^{p+1})\leq 2C_f(1+|s|^{p+1}).\label{3:4}
	\end{eqnarray}
	Based on \eqref{3:4} and on definition of $F(s)$, we have
	\begin{eqnarray}
		\begin{aligned}
			|F(s)|\leq \int_0^s|f(\tau)|d\tau\leq 2C_f\int_0^s(1+|\tau|^{p+1})d\tau\leq 2C_f\bigg(|s|+\frac{|s|^{p+2}}{p+2}\bigg)\leq\\ 2C_f\bigg(|s|^{p+2}+1+|s|^{p+2}\bigg).
		\end{aligned}
	\end{eqnarray}
	Considering $C_F=4C_f>0$, we arrive at
	\begin{eqnarray}
		|F(s)|\leq C_F(1+|s|^{p+2}),\quad\forall s\in\mathds{R}.\label{3:5}
	\end{eqnarray}
\end{remark}

In this section, we establish the existence and uniqueness of a global solution (in time) for  \eqref{p:1:1}-\eqref{p:1:3}. The main result of this section is Theorem \ref{thm:5:1}. First, let us consider the following Hilbert space
\begin{eqnarray}
	\pazocal{H}:=V_1\times V_0\times V_0,\label{4:1}
\end{eqnarray}
endowed with the inner product
\begin{eqnarray}
	\begin{aligned}
		\langle(u,\varphi, \theta),(\widetilde{u},\widetilde{\varphi},\widetilde{\theta})\rangle_{\pazocal{H}}:=&\ (A^{1/2} u,A^{1/2}\widetilde{u})_2
		+(\varphi,\widetilde{\varphi})_2+(\theta,\widetilde{\theta})_2
	\end{aligned}\label{4:2}
\end{eqnarray}
and the norm induced by \eqref{4:2}
\begin{eqnarray}
	\|(u,\varphi,\theta)\|^2_{\pazocal{H}}=\|A^{1/2} u\|_2^2+\|\varphi\|_2^2+\|\theta\|_2^2.\label{4:3}
\end{eqnarray}
Considering $U(t)=(u(t),u_t(t),\theta(t))$, $U_0=(u_0,u_1,\theta_0)$ and according to the previous settings, we can rewrite \eqref{p:1:1}-\eqref{p:1:3} in the form of following initial value problem in $\pazocal{H}$
\begin{eqnarray}
	\left\{\begin{array}{rcl}
		U_t(t)+\pazocal{A}U(t)&=&\pazocal{F}(U(t)),\quad t>0,\\
		U(0)&=&U_0\in \pazocal{H},
	\end{array}\right.\label{4:4}
\end{eqnarray}
where $\pazocal{A}:D(\pazocal{A})\subset\pazocal{H}\to\pazocal{H}$ and $\pazocal{F}:\pazocal{H}\to\pazocal{H}$ are defined by
\begin{eqnarray}
	\pazocal{A}\left(\begin{array}{c} u\\\varphi\\ \theta
	\end{array}\right)=\left(\begin{array}{c} -\varphi \\ Au+A^\nu\varphi-\delta A^\sigma\theta  \\ A\theta+\delta A^\sigma \varphi
	\end{array}\right)\label{4:5}
\end{eqnarray}
with
\begin{eqnarray}
	\begin{aligned}
		D(\pazocal{A}):=&\Big\{(u,\varphi,\theta)\in\pazocal{H};\ u,\theta \in H^2(\Omega),\ \varphi\in V_1\Big\},
	\end{aligned}\label{4:6}
\end{eqnarray}
and
\begin{eqnarray}
	\pazocal{F}\left(\begin{array}{c} u\\  \varphi\\ \theta
	\end{array}\right)=\left(\begin{array}{c} 0\\ h-f(u)\\ 0
	\end{array}\right).\label{4:7}
\end{eqnarray}

\begin{definition}
	A {\it strong solution} to \eqref{4:4} on $[0,T)$ is a continuous function $U:[0,T)\to \pazocal{H}$, continuously differentiable on $(0,T)$, $U(t)\in D(\pazocal{A})$ for any  $t\in (0,T)$ and \eqref{4:4} holds on $[0,T)$.
\end{definition}

\begin{definition}
	A {\it mild solution} to \eqref{4:4} on $[0,T]$ is a continuous function $U:[0,T]\to\pazocal{H}$ satisfying the following integral equation
	\begin{eqnarray}
		U(t)=T(t)U_0+\int_0^tT(t-s)\pazocal{F}(U(s))ds,\quad 0\leq t\leq T,\label{4:8}
	\end{eqnarray}
	where $T(t)$ is the semigroup generated by $\pazocal{A}$. As can be seen in \cite{doi:10.1007/978-1-4612-5561-1}[Chapter 6], any strong solution of \eqref{4:4} is also a mild solution.
\end{definition}

\begin{definition}
	A {\it weak solution} to \eqref{p:1:1}-\eqref{p:1:3} is a function $U=(u,u_t,\theta)\in C([0,\infty);\pazocal{H})$ with $U(0)=(u_0,u_1,\theta_0)$ such that the following identity is satisfied in the distributional sense 
	\begin{eqnarray}
		\begin{aligned}
			\frac{d}{dt}(u_t,\varphi)_2+(A^{1/2}u,A^{1/2}\varphi)_2+(A^\nu u_t,\varphi
			)_2-\delta(A^\sigma\theta,\varphi)_2+(f(u),\varphi)_2=(h,\varphi)_2,\\
			\frac{d}{dt}(\theta,\phi)_2+(A^{1/2}\theta,A^{1/2}\phi)_2+\delta (A^\sigma u_t,\phi)_2=0,
		\end{aligned}
	\end{eqnarray}
	for all $\varphi,\phi\in V_1$.
\end{definition}

\begin{remark}
	By adapting an argument found in \cite{doi:10.1007/978-1-4612-5561-1}[Chapter 4] it is possible to show that every mild solution on $[0, T] $ is the uniform limit on $[0,T']$, for every $0<T'<T$, of strong solutions to \eqref{4:4}. 
\end{remark}

\subsection{Energy of Solutions}

Let $U(t)=(u(t),u_t(t),\theta(t))$ be a strong solution of \eqref{4:4} on $[0,T)$. The {\it energy} $E=E(t)$ associated to $U(t)$ is defined by
\begin{eqnarray}
	E(t):=\frac{1}{2}\Big[\|A^{1/2} u(t)\|_2^2+\|u_t(t)\|_2^2+ \|\theta(t)\|_2^2\Big]= \frac{1}{2}\|U(t)\|^2_{\pazocal{H}},\label{4:9}
\end{eqnarray}
and the {\it modified energy} $\mathcal{E}(t)$ given by
\begin{eqnarray}
	\mathcal{E}(t):= E(t)+\int_\Omega F(u)dx-\int_\Omega hudx.\label{4:10}
\end{eqnarray}

\begin{lemma}\label{lem:4:1}
	The modified energy  \eqref{4:10} associated to solution $U(t)=(u(t),u_t(t),\theta(t))$ to \eqref{4:4} on $[0,T)$ is nonincreasing and satisfies
	\begin{eqnarray}
		\mathcal{E}(t)\geq \frac{1}{4}\|U(t)\|_{\pazocal{H}}^2-K,\quad t\in [0,T),\label{4:11}
	\end{eqnarray}
	and 
	\begin{eqnarray}
		\mathcal{E}(t)\leq M(\|U(t)\|_{\pazocal{H}}^{p+2} +1),\quad t\in [0,T),\label{4:11a}
	\end{eqnarray}
	for some positive constants $K$ and $M$ independent of $U$, $\nu$ and $\sigma$. As a consequence of \eqref{4:11}, one obtains 
	\begin{eqnarray}
		\lim_{t\to T^-}\|U(t)\|_{\pazocal{H}}<\infty.\label{4:12}
	\end{eqnarray}
\end{lemma}

\begin{proof}
	Suppose $U(t)$ is a strong solution. Taking the inner products (in $V_0$) of $\eqref{p:1:1}_1$ with $u_t$, $\eqref{p:1:1}_2$ with $\theta$ and adding the results, we obtain
	\begin{eqnarray}
		\frac{d\mathcal{E}}{dt}(t)=-\|A^{1/2}\theta\|_2^2-\|A^{\nu/2}u_t\|_2^2.
		\label{4:13}
	\end{eqnarray}
	Therefore, $\mathcal{E}(t)$ is nonincreasing, in particular
	\begin{eqnarray}
		\mathcal{E}(t)\leq\mathcal{E}(0),\quad t\in [0,T).\label{4:16}
	\end{eqnarray}
	To show \eqref{4:11}, we start by integrating \eqref{3:3} over $\Omega$ with respect to $x$, to get
	\begin{eqnarray}
		\begin{aligned}
			\int_\Omega F(u)dx\geq  -|\Omega| m_F.\label{4:17}
		\end{aligned}
	\end{eqnarray}
	Now, applying Young's inequality and \eqref{3:1a}, we obtain
	\begin{eqnarray}
		\begin{aligned}
			\int_\Omega  hudx\leq \kappa^2\|h\|_2^2+\frac{1}{4}\|A^{1/2} u\|_2^2. \label{4:18}
		\end{aligned}
	\end{eqnarray}
	Combining \eqref{4:10}, \eqref{4:17} and \eqref{4:18}, we arrive at
	\begin{eqnarray}
		\begin{aligned}
			\mathcal{E}(t)\geq&\   \frac{1}{4}\|A^{1/2}u\|_2^2+ \frac{1}{2}\|u_t\|_2^2+
			\frac{1}{2}\|\theta\|_2^2-\underbrace{\bigg[\kappa^2\|h\|_2^2+|\Omega|m_F\bigg]}_{K}.
		\end{aligned}\label{4:19}
	\end{eqnarray}
	From \eqref{4:19} we obtain \eqref{4:11}, and from \eqref{4:11} and \eqref{4:16} we obtain \eqref{4:12}. On the other hand, from \eqref{3:5}, we have 
	\begin{eqnarray}
		\int_\Omega F(u)dx\leq C_F(|\Omega|+\|u\|_{p+2}^{p+2}).\label{4:19a}
	\end{eqnarray}
	From embedding $V_1\hookrightarrow L^{p+2}(\Omega)$, there exists a positive constant $C_{|\Omega|}$, depending only on $\Omega$ and on $p+2$, such that
	\begin{eqnarray}
		\begin{aligned}
			C_F(|\Omega|+\|u\|_{p+2}^{p+2})\leq C_F(|\Omega|+C_{|\Omega|}\|A^{1/2}u\|_2^{p+2})\leq C_1(1+(\|A^{1/2}\|^2_2)^{(p+2)/2})\leq\\ C_1(1+\|U(t)\|^{p+2}_{\pazocal{H}}),
		\end{aligned}\label{4:19b}
	\end{eqnarray}
	where $C_1=\max\{C_F|\Omega|,C_FC_{|\Omega|}\}$. From \eqref{4:18}, we have
	\begin{eqnarray}
		\begin{aligned}
			\int_\Omega  hudx\leq \kappa^2\|h\|_2^2+\|A^{1/2} u\|_2^2\leq \kappa^2\|h\|^2+\|U(t)\|_{\pazocal{H}}^2\leq \kappa^2\|h\|^2+\|U(t)\|_{\pazocal{H}}^{p+2}+1\leq \\ C_2(1+\|U(t)\|_{\pazocal{H}}^{p+2})\label{4:19c}
		\end{aligned}
	\end{eqnarray}
	where $C_2=\kappa^2\|h\|^2+1$. Now, note that
	\begin{eqnarray}
		\|U(t)\|^2_{\pazocal{H}}\leq \|U(t)\|^{p+2}_{\pazocal{H}}+1.\label{4:19d}
	\end{eqnarray}
	From \eqref{4:10} and \eqref{4:19a}-\eqref{4:19d} and considering $M=C_1+C_2+1$ we obtain \eqref{4:11a}. Applying a density argument, we can prove that \eqref{4:11} and \eqref{4:11a} are valid for every mild solution on $[0, T']$ for any $0<T'<T$.
\end{proof}

\subsection{Existence of Solution}

\begin{lemma}\label{lem:5:1}
	The operator $\pazocal{A}$ is m-accretive.
\end{lemma}
\begin{proof}
	From definition of $\pazocal{A}$ in \eqref{4:5} and \eqref{4:2}, we have
	\begin{eqnarray}
		\langle\pazocal{A}U,U\rangle_{\pazocal{H}}=\|A^{\nu/2}\varphi\|_2^2+\|A^{1/2}\theta\|_2^2,\quad\forall U=(u,\varphi,\theta)\in D(\pazocal{A}).\label{5:1}
	\end{eqnarray}
	Moreover, we have that $R(I+\pazocal{A})=\pazocal{H}$. In order to prove this, given $F=(u^*,\varphi^*,\theta^*)\in\pazocal{H}$, we must obtain $U=(u,\varphi,\theta)\in D(\pazocal{A})$, such that
	\begin{eqnarray}
		(I+\pazocal{A})U=U^*.\label{5:2}
	\end{eqnarray}
	The previous inequality is equivalent to the following system
	\begin{eqnarray}
		u-\varphi&=&u^*,\label{5:2a}\\
		\varphi +Au+A^\nu\varphi-\delta A^\sigma\theta&=&\varphi^*,\label{5:2b}\\
		\theta +A\theta +\delta A^\sigma\varphi&=&\theta^*,\label{5:2c}
	\end{eqnarray}
	By using \eqref{5:2a} into \eqref{5:2b} and into \eqref{5:2c}, we obtain
	\begin{eqnarray}
		u +A u+A^\nu u-\delta A^\sigma\theta&=&\varphi^*+u^*+A^\nu u^*,\label{5:3a}\\
		\theta +A\theta +\delta A^\sigma u&=&\theta^*+\delta A^\sigma u^*,	\label{5:3b}
	\end{eqnarray}
	Note that the right sides of \eqref{5:3a}-\eqref{5:3b} belong to $V_0$. To solve \eqref{5:3a}-\eqref{5:3b} we use a standard variational approach in order to obtain the bilinear functional $\pazocal{G}:(V_1\times V_1)^2\to\mathbb{R}$ given by
	\begin{eqnarray}
		\begin{aligned}
			\pazocal{G}\Big((u,\theta),(\widetilde{u},\widetilde{\theta})\Big):=&\ (u,\widetilde{u})_2+(A^{1/2} u,A^{1/2} \widetilde{u})_2+(A^{\nu/2} u,A^{\nu/2}\widetilde{u})_2-\delta(A^{\sigma/2}\theta,A^{\sigma/2}\widetilde{u})_2+\\ &\ (\theta,\widetilde{\theta})_2+(A^{1/2}\theta,A^{1/2}\widetilde{\theta})_2+\delta(A^{\sigma/2} u,A^{\sigma/2}\widetilde{\theta})_2.
		\end{aligned}
	\end{eqnarray}
	By using Holder's inequality and \eqref{3:1a} it is not difficult to verify that $\pazocal{G}$ is continuous. In addition 
	\begin{eqnarray}
		\begin{aligned}
			\pazocal{G}\Big((u,\theta),(u,\theta)\Big):=&\ \|u\|_2^2+\|A^{1/2} u\|_2^2+\|A^{\nu/2} u\|_2^2+\|\theta\|_2^2+ \|A^{1/2}\theta\|_2^2,
		\end{aligned}
	\end{eqnarray}
	which shows that $\pazocal{G}$ is coercive. Therefore, it follows from the Lax-Milgram's  Theorem that the system \eqref{5:3a}-\eqref{5:3b} has a unique solution $(u,\theta)\in V_1\times V_1$. From \eqref{5:2a}  we have $\varphi\in V_1$, from \eqref{5:3b}, we have $A\theta\in V_0$ and thus $\theta\in H^2(\Omega)$, from \eqref{5:3a}, we have $A u\in L^2(\Omega)$ and thus $u\in H^2(\Omega)$. Therefore  $(u,\varphi,\theta)\in D(\pazocal{A})$ and the result follows from \cite{ISBN:9781584884521}[Lemma 2.2.3].
\end{proof}

\begin{lemma}\label{lem:5:4}
	The operator $\pazocal{F}$ is locally Lipschitz.
\end{lemma}
\begin{proof}
	Let $U=(u,\varphi,\theta)$,  $\widetilde{U}=(\widetilde{u},\widetilde{\varphi},\widetilde{\theta})$ in $\pazocal{H}$ and $\overline{K}>0$ such that 	
	\begin{eqnarray}
		\|U\|_{\pazocal{H}},\|\widetilde{U}\|_{\pazocal{H}}\leq\overline{K}.\label{5:8}
	\end{eqnarray}	
	Throughout the proof, $C$ and $C(\overline{K})$ represent generic positive constants independent of $K$ and depending on $\overline{K}$, respectively. From \eqref{4:7}, we have
	\begin{eqnarray}
		\|\pazocal{F}(U)-\pazocal{F}(\widetilde{U})\|^2_{\pazocal{H}}=\int_\Omega|f(u)-f(\widetilde{u})|^2dx.\label{5:9}
	\end{eqnarray}
	By using Mean Value Theorem, Schwarz's inequality and \eqref{3:2}, for any $x\in\Omega$ there exist $\vartheta=\vartheta(x)\in(0,1)$, such that
	\begin{eqnarray}
		\begin{aligned}
			\int_\Omega|f(u)-f(\widetilde{u})|^2dx \leq \int_\Omega| f'((1-\vartheta)u+\vartheta \widetilde{u})|^2|u-\widetilde{u}|^2dx\leq\\  C\int_\Omega (1+|(1-\vartheta)u+\vartheta \widetilde{u})|^p)^2|u-\widetilde{u}|^2dx
			\leq C\int_\Omega(1+|u|^p+|\widetilde{u}|^p)^2|u-\widetilde{u}|^2dx.
		\end{aligned}\label{5:10}
	\end{eqnarray}
	It follows from Hölder's inequality
	\begin{eqnarray}
		\begin{aligned}
			\int_\Omega(1+|u|^p+|\widetilde{u}|^p)^2|u-\widetilde{u}|^2dx \leq   C\bigg(\int_\Omega(1+|u|^p+|\widetilde{u}|^p)^3\bigg)^{2/3}\bigg(\int_\Omega|u-\widetilde{u}|^6\bigg)^{1/3}
			\leq \\ C(1+\|u\|_{3p}+\|\widetilde{u}\|_{3p})^{2p}\|u-\widetilde{u}\|_6^{2}.
		\end{aligned}\label{5:11}
	\end{eqnarray}
	From Sobolev embedding (in 2D) $V_1\hookrightarrow L^s(\Omega)$ for $0\leq s < \infty$ and \eqref{5:8}, we have
	\begin{eqnarray}
		\begin{aligned}
			(1+\|u\|_{3p}+\|\widetilde{u}\|_{3p})^{2p}\|u-\widetilde{u}\|_6^{2}\leq C(1+\|A^{1/2} u\|_2+\|A^{1/2} \widetilde{u}\|_2)^{2p}\|A^{1/2} u-A^{1/2} \widetilde{u}\|_2^{2}\leq\\  C(\overline{K})\|U-\widetilde{U}\|_{\pazocal{H}}^2,\label{5:12}
		\end{aligned}
	\end{eqnarray}
	where $C(\overline{K})$ is a positive constant depending on $K$. Combining \eqref{5:9}-\eqref{5:12}, we obtain the result.
\end{proof}

\begin{theorem}[{\bf Existence of Global Solution}]\label{thm:5:1}
	Regarding the problem \eqref{4:4} we have:
	\begin{description}
		\item[a)] If $U_0\in\pazocal{H}$ then \eqref{4:4} has a unique global (in time) mild solution;
		\item[b)] If $U_0\in D(\pazocal{A})$, then the mild solution obtained in (a) is strong solution;
		\item[c)] If $U^1(t)$ and $U^2(t)$ are two mild solutions to \eqref{4:4}, then there exists a positive constant $C_0=C_0(U^1(0),U^2(0))$ depending on $U^1(0)$ and on $U^2(0)$  such that, for every $T>0$ we have
		\begin{eqnarray}
			\|U^1(t)-U^2(t)\|_{\pazocal{H}}\leq e^{C_0 t}\|U^1(0)-U^2(0)\|_{\pazocal{H}},\quad 0\leq t\leq T.\label{5:13}
		\end{eqnarray}
	\end{description}
\end{theorem}
\begin{proof}
	{\bf (a)-(b).} Since $\pazocal{A}$ is m-accretive and $\pazocal{F}$ is local Lipschitz, It follows from \cite{ISBN:9781584884521}[Theorem 2.5.4] that there exist $T_{\text{max}}>0$ such that if $U_0\in\pazocal{H}$ then \eqref{4:4} has a unique mild solution on $[0,T_{\text{max}})$ and if $U_0\in D(\pazocal{A})$ then \eqref{4:4} has a unique strong solution on $[0,T_{\text{max}})$, in addition, \eqref{4:12} holds. Therefore, it follows from \cite{ISBN:9781584884521}[Theorem 2.5.5] that $T_{\text{max}}=\infty$.\\
	
	{\bf (c).} Using a standard procedure that consists of considering the norm of the difference between two mild solutions, the local Lipschitz continuity of $\pazocal{F}$ and then applying the Grönwall's inequality (see \cite{doi:10.1007/978-1-4612-5561-1}[Section 6.1]), it is possible to obtain \eqref{5:13}.
\end{proof}

\begin{remark}\label{rem:3:1a}
	
	A (global) mild solution $U(t)$ to \eqref{4:4} produces a nonlinear semigroup $S_\beta(t)$, $\beta=(\nu,\sigma)\in\Lambda$ on $\pazocal{H}$ given by
	\begin{eqnarray}
		S_\beta(t)U_0:=U(t)=(u(t),u_t(t),\theta(t)), \quad U_0=(u_0,u_1,\theta_0)\in\pazocal{H}.\label{5:18}
	\end{eqnarray}
	Thanks to the regularity of the solution $U$, we have $S_\beta(t)$ a $C_0$-semigroup,  therefore $(\pazocal{H},S_\beta(t))$ is a dynamical system.
\end{remark}

\section{Global Attractor and Finite Dimensionality}

The goal of this section is to prove that the dynamical system $(\pazocal{H},S_\beta(t))$, given in Remark \ref{rem:3:1a}, possesses a compact global attractor (see Definition \ref{def:AR:2}). The main result of this section is

\begin{theorem}\label{thm:4:1}
	The dynamical system $(\pazocal{H},S_\beta(t))$ possesses a compact global attractor $\mathfrak{A}_\beta=\mathcal{M}(\mathcal{N}_\beta)$, where $\mathcal{M}(\mathcal{N}_\beta)$ is the  unstable manifold emanating from set of stationary points $\mathcal{N}_\beta$ of $(\pazocal{H},S_\beta(t))$ (see Definition \ref{def:AR:4}).
\end{theorem}

\subsection{Quasi-Stability}

The property of a dynamical system being quasi-stable on bounded positively invariant sets ensures it is asymptotically smooth (Theorem \ref{thm:ar:1:2}). 

If we consider \eqref{5:18}, $X=Z=V_1$, $Y=V_0$, we obtain $\pazocal{H}=X\times Y\times Z$, with $X$  compactly embedded in $Y$. In addition,
\begin{eqnarray}
	u\in C(\mathds{R}_+;X)\cap C^1(\mathds{R}_+;Y)\quad \text{and}\quad \theta\in C(\mathds{R}_+,Z).\label{6:5}
\end{eqnarray}
Therefore, the structures of the phase space $\pazocal{H}$ and the evolution operator $S_\beta(t)$ satisfy the Assumption \ref{ar:assump:1}. The main result of this subsection is given by the following result:

\begin{theorem}\label{thm:6:4}
	The dynamical system $(\pazocal{H},S_\beta(t))$ is quasi-stable on every bounded positively invariant set $\pazocal{O}\subset\pazocal{H}$ and therefore is asymptotically smooth. More precisely, if $S_\beta(t)U_0^i=(u^i,u_t^i,\theta^i)$ is mild solution of \eqref{p:1:1}-\eqref{p:1:3} with initial data $U_0^i\in\pazocal{O}$, $i=1,2$. Then, there exist nonnegative functions $a,b,c:\mathds{R}_+\to\mathds{R}$ satisfying Definition \ref{ar:def:1:1}{\bf (a)-(b)} and a compact seminorm $\mu$ on $X$  such that
	\begin{eqnarray}
		\|S_\beta(t)U_0^1-S_\beta(t)U_0^2\|_{\pazocal{H}}\leq a(t)\|U^1_0-U^2_0\|_{\pazocal{H}}\label{6:5b}
	\end{eqnarray}
	and
	\begin{eqnarray}
		\|S_\beta(t)U_0^1-S_\beta(t)U_0^2\|^2_{\pazocal{H}}\leq b(t)\|U_0^1-U_0^2\|_{\pazocal{H}}^2 + c(t)\sup_{0\leq s\leq t}[\mu(u^1(s)-u^2(s))]^2.\label{6:5a}
	\end{eqnarray}
\end{theorem}
\begin{proof}
	First, by taking $a(t):=e^{C_0t}$ in \eqref{5:13}, we get \eqref{6:5b}. Now, let us consider the following settings, for $i=1,2$, consider $U^i_0=(u^i_0,u_1^i,\theta_0^i)$ $\in \pazocal{O}$, such that
	\begin{eqnarray}
		U^i(t):=S_\beta(t)U_0^i=(u^i(t),u_t^i(t),\theta^i(t)),\label{6:5:1}
	\end{eqnarray}
	is strong solution of \eqref{p:1:1}, so
	\begin{eqnarray}
		U=U^1-U^2=(u,u_t,\theta),\label{6:6}
	\end{eqnarray}
	where $u=u^1-u^2$ and $\theta=\theta^1-\theta^2$, satisfies 
	\begin{eqnarray}
		\begin{aligned}
			u_{tt}+A u+A^\nu u_t-\delta A^\sigma \theta=&\ -G(u),\\
			\theta_t- A\theta+\delta A^\sigma u_{t}=&\ 0,
		\end{aligned}\label{5:14}
	\end{eqnarray}
	with boundary conditions 
	\begin{eqnarray}
		u=\theta=0\quad\text{on}\quad \Gamma\times(0,\infty),
	\end{eqnarray}
	and initial conditions 
	\begin{eqnarray}
		(u(0),u_t(0),\theta(0))=U_0^1-U_0^2,
	\end{eqnarray}
	where
	\begin{eqnarray}
		G(u)=f(u^1)-f(u^2).
	\end{eqnarray}
	
	\hspace{\oddsidemargin}{\bf Step 1.} We claim that there exists a constant $C_{\pazocal{O}}>0$ dependent only on $\pazocal{O}$, such that
	\begin{eqnarray}
		\frac{dE}{dt}(t)\leq -\frac{1}{2}\|A^{\nu/2}u_t\|_2^2 -\|A^{1/2}\theta\|_2^2+C_{\pazocal{O}}\|u\|^2_6.\label{6:9}
	\end{eqnarray}
	Multiplying $\eqref{5:14}_1$ by $u_t$, $\eqref{5:14}_2$ by $\theta $, integrating over $\Omega$ and by using \eqref{3:1a}, we obtain
	\begin{eqnarray}
		\frac{dE}{dt}(t)=-\|A^{\nu/2}u_t\|_2^2-\|A^{1/2}\theta\|_2^2-\int_\Omega G(u)u_tdx.\label{6:9a}
	\end{eqnarray}
	Now, applying the Young's inequality and \eqref{3:1a}, we get
	\begin{eqnarray}
		\begin{aligned}
			\bigg|\int_\Omega G(u)u_tdx\bigg| \leq \int_\Omega|f(u^1)-f(u^2)||u_t|dx
			\leq &\ C_\kappa\int_\Omega|f(u^1)-f(u^2)|^2dx+\frac{\kappa^{-2}}{2}\|u_t\|_2^2\\
			\leq &\ C_\kappa\int_\Omega|f(u^1)-f(u^2)|^2dx+\frac{1}{2}\|A^{\nu}u_t\|_2^2,
		\end{aligned}\label{6:9b}
	\end{eqnarray}
	where $C_\kappa$ is a positive constant (depending on $\kappa$). Performing a calculation analogous to \eqref{5:10}-\eqref{5:11}, we obtain
	\begin{eqnarray}
		\int_\Omega|f(u^1)-f(u^2)|^2dx\leq C(1+\|u^1\|_{3p}+\|u^2\|_{3p})^{2p}\|u\|_6^{2}.
	\end{eqnarray}
	Again using the Sobolev embedding mentioned in Lemma \ref{lem:5:4}, we arrived at
	\begin{eqnarray}
		\int_\Omega|f(u^1)-f(u^2)|^2dx\leq C(1+\|A^{1/2} u^1\|^2_2+\|A^{1/2} u^2\|^2_2)^p\|u\|_6^{2}\leq C_{\pazocal{O}}\|u\|_6^{2},\label{6:9c}
	\end{eqnarray}
	where $C_{\pazocal{O}}$ is a positive constant dependent only on $\pazocal{O}$ (because $\pazocal{O}$ is positively invariant).
	Combining \eqref{6:9a}, \eqref{6:9b} and \eqref{6:9c}, we get \eqref{6:9}.\\
	
	\hspace{\oddsidemargin}{\bf Step 2.} For function $\pazocal{J}=\pazocal{J}(t)$, given by
	\begin{eqnarray}
		\pazocal{J}(t)=(u_t,u)_2.\label{6:10}
	\end{eqnarray}
	The following inequality holds
	\begin{eqnarray}
		\begin{aligned}
			\pazocal{J}'(t)\leq &\ -\frac{1}{2}\|A^{1/2} u\|_2^2+\frac{5\kappa^2}{2}\|A^{\nu/2} u_t\|_2^2+\frac{3\kappa^4}{2}\|A^{1/2} \theta\|_2^2 +C_{\pazocal{O}}\|u\|^2_6,\label{6:11}
		\end{aligned}
	\end{eqnarray}
	where $C_1>0$ and $C_2>0$ are constants. In fact, taking the derivative of $\pazocal{J}$ and using $\eqref{5:14}_1$, we get
	\begin{eqnarray}
		\pazocal{J}'(t)=-\|A^{1/2} u\|_2^2-(A^{\nu/2} u_t,A^{\nu/2}u)_2+\delta(A^{\sigma/2}\theta,A^{\sigma/2} u)_2-\int_\Omega G(u)udx+\|u_t\|_2^2.\label{6:12}
	\end{eqnarray}
	Applying the Hölder's and Young's inequalities, we have 
	\begin{eqnarray}
		\begin{aligned}
			\pazocal{J}'(t)\leq  -\|A^{1/2} u\|_2^2+\bigg(\frac{3\kappa^2}{2}\|A^{\nu/2} u_t\|_2^2+\frac{1}{6\kappa^2}\|A^{\nu/2}u\|_2^2\bigg)+\\
			\bigg(\frac{3\kappa^2}{2}\|A^{\sigma/2} \theta\|_2^2+\frac{1}{6\kappa^2}\|A^{\sigma/2} u\|_2^2 \bigg)
			 -\int_\Omega G(u)udx+ \|u_t\|_2^2.\label{6:13}
		\end{aligned}
	\end{eqnarray}
	By using \eqref{3:1a} in \eqref{6:13}, we obtain
	\begin{eqnarray}
		\begin{aligned}
			\pazocal{J}'(t)\leq  -\|A^{1/2} u\|_2^2+\bigg(\frac{3\kappa^2}{2}\|A^{\nu/2} u_t\|_2^2+\frac{1}{6}\|A^{1/2}u\|_2^2\bigg)+\\
			\bigg(\frac{3\kappa^4}{2}\|A^{1/2} \theta\|_2^2+\frac{1}{6}\|A^{1/2} u\|_2^2 \bigg) 
			-\int_\Omega G(u)udx+\kappa^2\|A^{\nu/2}u_t\|_2^2.\label{6:13a}
		\end{aligned}
	\end{eqnarray}  
	Performing a procedure analogous to \eqref{6:9b}-\eqref{6:9c} and applying the Poincaré's  inequality, we have
	\begin{eqnarray}
		\int_\Omega G(u)udx\leq C_{\pazocal{O}}\|u\|_6^2+\frac{1}{6}\|A^{1/2} u\|_2^2\label{6:14}.
	\end{eqnarray}
	Combining \eqref{6:13a} and \eqref{6:14}, we arrive at \eqref{6:11}.
	\\
	
	\hspace{\oddsidemargin}{\bf Step 3.} Consider the funcional
	\begin{eqnarray}
		\pazocal{L}(t)=N_1E(t)+N_2\pazocal{J}(t),\label{6:15}
	\end{eqnarray}
	where $N_1$ and $N_2$ are positive constants to be established later. By using Young's and Poincaré's inequalities in \eqref{6:10}, we can show that there exists a positive constant $\gamma_0$ such that
	\begin{eqnarray}
		|\pazocal{L}(t)-N_1E(t)|\leq \gamma_0 E(t),\label{6:16}
	\end{eqnarray}
	thus
	\begin{eqnarray}
		(N_1-\gamma_0)E(t)\leq \pazocal{L}(t)\leq (N_1+ \gamma_0) E(t).\label{6:17}
	\end{eqnarray}
	From \eqref{6:9}, \eqref{6:11} and \eqref{6:15}, we have
	\begin{eqnarray}
		\begin{aligned}
			\pazocal{L}'(t)=&\ N_1E'(t)+N_2\pazocal{J}'(t)\\
			\pazocal{L}'(t)\leq&\ N_1\bigg(-\frac{1}{2}\|A^{\nu/2}u_t\|_2^2 -\|A^{1/2}\theta\|_2^2+C_{\pazocal{O}}\|u\|^2_6\bigg)+\\
			&\ N_2\bigg(-\frac{1}{2}\|A^{1/2} u\|_2^2+\frac{5\kappa^2}{2}\|A^{\nu/2} u_t\|_2^2+\frac{3\kappa^4}{2}\|A^{1/2} \theta\|_2^2 +C_{\pazocal{O}}\|u\|^2_6\bigg)\\
			\leq &\ -\bigg(\frac{N_1}{2}-\frac{5N_2\kappa^2}{2}\bigg)\|A^{\nu/2}u_t\|_2^2-\frac{N_2}{2}\|A^{1/2} u\|_2^2-\Big(N_1-\frac{3N_2\kappa^4}{2}\Big)\|A^{1/2}\theta\|_2^2+ \\
			&\ \Big(N_1+N_2\Big)C_{\pazocal{O}}\|u\|^2_6.
		\end{aligned}\label{6:18}
	\end{eqnarray}
	Now, taking $N_1$ sufficiently large and taking into account \eqref{3:1a},  we obtain positive constants $\gamma_1$, $\gamma_2$ and $\gamma_3$ such that
	\begin{eqnarray}
		\gamma_1E(t)\leq \pazocal{L}(t)\leq \gamma_2 E(t)\label{6:19}
	\end{eqnarray}
	and
	\begin{eqnarray}
		\pazocal{L}'(t)\leq -\gamma_3E(t)+C_{\pazocal{O}}\|u\|_6^2.\label{6:20}
	\end{eqnarray}
	By using \eqref{6:19} and \eqref{6:20}, we arrive at
	\begin{eqnarray}
		\pazocal{L}'(t)\leq -\frac{\gamma_3}{\gamma_2}\pazocal{L}(t)+C_{\pazocal{O}}\|u\|_6^2.\label{6:20a}
	\end{eqnarray}
	From inequality above, we obtain
	\begin{eqnarray}
		\frac{d}{dt}\Big\{e^{\frac{\gamma_3}{\gamma_2}t}\pazocal{L}(t)\Big\}\leq C_{\pazocal{O}} e^{\frac{\gamma_3}{\gamma_2}t} \|u\|_6^2.\label{6:20b}
	\end{eqnarray}
	Integrating the inequality above over $[0,t]$, we get
	\begin{eqnarray}
		\pazocal{L}(t)\leq \pazocal{L}(0)e^{-\frac{\gamma_3}{\gamma_2}t} + C_{\pazocal{O}} \int_0^te^{\frac{\gamma_3}{\gamma_2}(s-t)}\|u(s)\|_6^2ds.\label{6:20c}
	\end{eqnarray}
	By using \eqref{6:19} into \eqref{6:20c} we obtain
	\begin{eqnarray}
		E(t)\leq\frac{\gamma_2}{\gamma_1}E(0)e^{-\frac{\gamma_3}{\gamma_2}t} + \frac{C_{\pazocal{O}}}{\gamma_1} \int_0^te^{\frac{\gamma_3}{\gamma_2}(s-t)}ds\cdot\sup_{0\leq s\leq t}[\mu(u(s))]^2,\label{6:23}
	\end{eqnarray}
	where $\mu$ is the seminorm on $V_1$ defined by
	\begin{eqnarray}
		\mu (z)=\|z\|_6.
	\end{eqnarray}
	From \eqref{6:5:1}, \eqref{6:6} and \eqref{6:23}, we obtain
	\begin{eqnarray}
		\begin{aligned}
			\|S_\beta(t)U_0^1-S\nu(t)U_0^2\|^2_{\pazocal{H}}\leq\frac{\gamma_2}{\gamma_1}\|U_0^1-U_0^2\|_{\pazocal{H}}^2e^{-\frac{\gamma_3}{\gamma_2}t} +\\
			\frac{2C_{\pazocal{O}}}{\gamma_1} \int_0^te^{\frac{\gamma_3}{\gamma_2}(s-t)}ds\cdot\sup_{0\leq s\leq t}[\mu(u^1(s)-u^2(s))]^2.
		\end{aligned}
	\end{eqnarray}
	So, by taking 
	\begin{eqnarray}
		b(t):=\frac{\gamma_2}{\gamma_1}e^{-\frac{\gamma_3}{\gamma_2}t} \quad\text{and}\quad  c(t):= \frac{C_{\pazocal{O}}}{\gamma_1} \int_0^te^{\frac{\gamma_3}{\gamma_2}(s-t)}ds,\label{6:23:1}
	\end{eqnarray}
	we obtain \eqref{6:5a}. It is worth mentioning that, using a density argument, the inequality \eqref{6:5a} holds also for mild solutions at $[0,T]$, for all $T>0$.
\end{proof}

\begin{remark}\label{rem:4:3}
	Note that the constants obtained in Theorem \ref{thm:6:4} independ of $\nu$ and $\sigma$ since embedding constants for $\nu,\sigma \in(0,1/2)$ are uniform by \eqref{3:1a}.
\end{remark}

\subsection{Gradient System}

By Definition \ref{def:AR:4}, the set $\mathcal{N}_\beta$, mentioned in Theorem \ref{thm:4:1} is defined as
\begin{eqnarray}
	\mathcal{N}_\beta=\{U_0 \in\pazocal{H}; \ S_\beta(t)U_0=U_0, \ \forall t \geq
	0\}.\label{6:23a}
\end{eqnarray}
Therefore, if $U_0=(u_0,u_1,\theta_0)\in\mathcal{N}_\beta $, then from \eqref{4:8}
\begin{eqnarray}
	U_0=T(t)U_0+\int_0^tT(t-s)\pazocal{F}(U_0)ds,
\end{eqnarray}
and for every $t>0$
\begin{eqnarray}
	-\frac{(T(t)-I)}{t}U_0=\frac{1}{t}\int_0^tT(\tau)\pazocal{F}(U_0)d\tau.
\end{eqnarray}
The limit as $t\to 0$ on the right side of the previous equality exists and is equal to $\pazocal{F}(U_0) $, so $U_0\in D(\pazocal{A}) $, and satisfies
\begin{eqnarray}
	\pazocal{A}U_0+\pazocal{F}(U_0)=0.\label{6:23b}
\end{eqnarray}
Considering \eqref{4:5}, \eqref{4:7} and \eqref{6:23b}, we arrive at
\begin{eqnarray}
	\mathcal{N}_\beta=\{(u,0,0)\in D(\pazocal{A})\ |\   
	A u+ f(u)= h\},\label{6:23c}
\end{eqnarray}
in this case, $\mathcal{N}_\beta$ is independent of $\beta$. Based on \eqref{6:23c} we obtain the following lemma
\begin{lemma}
	The set of stationary points $\mathcal{N}_\beta$ of $(\pazocal{H},S_\beta(t))$ is bounded.
\end{lemma}
\begin{proof}
	if $(u,0,0)\in\mathcal{N}_\alpha$ then $u\in V_2$ satisfies
	\begin{eqnarray}
		A u+f(u)=h.\label{6:2:3}
	\end{eqnarray}
	Multiplying $\eqref{6:2:3}$ by $u$ integrating over $\Omega$,  considering \eqref{3:3} and applying Young's and Poincaré's inequalities, we obtain
	\begin{eqnarray}
		\|A^{1/2} u\|_2^2 = -\int_\Omega f(u)udx+(h,u)_2\leq \frac{\kappa^2}{2}\|h\|_2^2+\frac{1}{2}\|A^{1/2}u\|_2^2,
	\end{eqnarray}
	thus,
	\begin{eqnarray}
		\|(u,0,0)\|^2_{\pazocal{H}}=\|A^{1/2} u\|_2^2 \leq  \kappa^2\|h\|_2^2.\label{6:2:3a}
	\end{eqnarray}
	Therefore, $\mathcal{N}_\beta$ is bounded.
\end{proof}

Let us now consider the function $\Phi:\pazocal{H}\to\mathds{R}$, define by
\begin{eqnarray}
	\Phi(u,\varphi,\theta)=\frac{1}{2}\|(u,\varphi,\theta)\|^2_{\pazocal{H}}+\int_\Omega F(u)dx-(h,u)_2.\label{6:2:6}
\end{eqnarray}
Note that $\Phi$ is continuous because $F$ is continuous (Assumption \ref{asp:3:2}). From Definition \ref{def:AR:5}, we have
\begin{lemma}
	The dynamical system $(\pazocal{H},S_\beta(t))$ is gradient, with $\Phi$ given in \eqref{6:2:6} a strict Lyapunov function for $(\pazocal{H},S_\beta(t))$.
\end{lemma}
\begin{proof}
	If $S_\beta(t)U_0=(u(t),u_t(t),\theta(t))$ is a strong solution of \eqref{4:4}, then from Lemma \ref{lem:4:1}, for any $U_0=(u_0,u_1,\theta_0)\in\pazocal{H}$, $t\mapsto\Phi(S_\beta(t)U_0)$ is nonincreasing. In addition, if $\Phi(S_\beta(t)U_0)=\Phi(U_0)$ for any $t>0$, then from \eqref{4:13} and \eqref{3:1a}, we have
	\begin{eqnarray}
		\begin{aligned}
			0=\Phi(S(t)U_0)-\Phi(U_0)=& -\int_0^t\|A^{\nu/2}u_t\|_2^2d\tau -\int_0^t\|A^{1/2}\theta\|_2^2d\tau\leq 0,
		\end{aligned}\label{6:2:7}
	\end{eqnarray}
	thus
	\begin{eqnarray}
		\|A^{\nu/2}u_t(t)\|_2=\|A^{1/2} \theta (t)\|_2=0.
	\end{eqnarray}
	From \eqref{3:1a}, we have
	\begin{eqnarray}
		u_t(t)=\theta (t)=0,\quad t\geq 0\quad\text{so}\quad  u(t)=u_0,\quad t>0.
	\end{eqnarray}
	Therefore, $U_0=(u_0,0,0)\in\mathcal{N}_\beta$. This means that $\Phi$ is a strict  Lyapunov function for $(\pazocal{H},S_\beta(t))$ on $\pazocal{H}$. Using a density argument, it is possible to establish that the inequality \eqref{6:2:7} is valid for mild solutions on $[0,T]$ for all $T>0$.
\end{proof}

\begin{lemma}
	The Lyapunov function \eqref{6:2:6} satisfies the following conditions:
	\begin{description}
		\item[a)] $\Phi$ is bounded from above on any bounded subset of $\pazocal{H}$;
		\item[b)] The set $\Phi_R=\{W\in\pazocal{H}; \ \Phi(W)\leq R\}$ is bounded.
	\end{description}
\end{lemma}
\begin{proof}
	{\bf(a)} Since $(\pazocal{H},S_\beta(t))$ is gradient with Lyapunov function $\Phi=\mathcal{E}$, it follows from definition of $\mathcal{E}(t)$ that $\Phi$ is bounded from above on every bounded subset of $\pazocal{H}$.
	
	{\bf(b)} For every $W\in \Phi_R$, it follows from definition of $\Phi_R$ and $\Phi$, and from \eqref{4:11} that
	\begin{eqnarray}
		\|W\|^2_{\pazocal{H}}\leq 4\Phi(W)+K\leq 4R+K.
	\end{eqnarray}
	Therefore, $\Phi_R$ is a bounded set.
\end{proof}

\begin{proof}[{\bf Proof of Theorem \ref{thm:4:1}}]
	Since $(\pazocal{H},S_\beta(t))$ is gradient, asymptotically smooth, $\Phi$ is bounded above on any bounded subset of $\pazocal{H}$, $\Phi_R$ and $\mathcal{N}_{\beta}$ are  bounded set, the result follows from Theorem \ref{thm:ar:1:1}.
\end{proof}

\begin{remark}\label{rem:4:5}
	Consider the Lyapunov function $\Phi$ given in \eqref{6:2:6}. By using \eqref{4:11}, \eqref{4:11a} and Remark 7.5.8 in \cite{ISBN:9780387-877112}, we have
	\begin{eqnarray}
		\begin{aligned}
			\sup_{W\in\mathfrak{A}_{\beta}}\|W\|^2_{\pazocal{H}}\leq 4\bigg(\sup_{W\in\mathfrak{A}_{\beta}}\Phi(W)+K\bigg)\leq 4\bigg(\sup_{W\in\mathcal{N}_{\beta}}\Phi(W)+K\bigg)\leq 4M\sup_{W\in\mathcal{N}_{\beta}}\|W\|^{p+2}_{\pazocal{H}}\\
			+4(M+K).
		\end{aligned}
	\end{eqnarray}
	Since that $\mathcal{N}_{\beta}$ is a bounded set (see \eqref{6:2:3a}), we conclude that there exists a constant $C_0>0$, independent of $\nu$, such that 
	\begin{eqnarray}
		\sup_{W\in\mathfrak{A}_{\beta}}\|W\|^2_{\pazocal{H}}\leq C_0.
	\end{eqnarray}
	Therefore, by taking $R_0^2=1+C_0$ and considering $\mathcal{B}_0=\overline{B}(0,R_0)\subset\pazocal{H}$, the closed ball centered at the origin and radius $R_0$, we obtain a bounded absorbing set (independent of $\nu$) for $(\pazocal{H},S_\beta(t))$. Thus there exists (a time) $t_{\mathcal{B}_0}>0$ such that   $S_\beta(t)\mathcal{B}_0\subset\mathcal{B}_0$ for any $t\geq t_{\mathcal{B}_0}$. As a consequence, if we consider $\mathcal{B}=\cup_{t\geq t_{\mathcal{B}_0}}S(t)\mathcal{B}_0$, then $\mathcal{B}$ is a bounded positively invariant absorbing set (independent of $\nu\in(0,1/2)$) with $\mathcal{B}\subset\mathcal{B}_0$.
\end{remark}

\begin{theorem}
	The global attractor $\mathfrak{A}_\nu$ has finite fractal dimension.
\end{theorem}
\begin{proof}
	Since the attractor $\mathfrak{A}_\nu$ is bounded and invariant, from Theorem \ref{thm:6:4},  $(\pazocal{H},S_\beta(t))$ is quasi-stable on $\mathfrak{A}_\nu$, therefore, follows from the Theorem \ref{thm:ar:1:3} that $\mathfrak{A}_\nu$ has a finite fractal dimension.
\end{proof}

\section{Regularity of trajetories  and exponential attractors}
\begin{theorem}\label{thm:5:1a}
	Every complete trajectory $\gamma=\{(u(t), u_t(t),\theta(t)): t\in\mathds{R}\}$
	in $\mathfrak{A}$ has further regularity
	\begin{eqnarray}
		\|u(t)\|_{V_2}+\|u_t(t)\|_{V_1}+\|\theta(t)\|_{V_2}\leq
		R,\quad\forall t\in\mathds{R},\label{8:1}
	\end{eqnarray}
	for some $R>0$. Therefore, $\mathfrak{A}$ is uniformly bounded in $\pazocal{H}_1:=V_2\times V_1\times V_2$.
\end{theorem}
\begin{proof}
	Since the dynamical system $(\pazocal{H},S(t))$ is quasi-stable on bounded subset of $\pazocal{H}$, $(\pazocal{H},S(t))$ is quasi-stable on the attractor
	$\mathfrak{A}$. Note that $\sup_{t\in\mathds{R}_+}c(t)<\infty$ in \eqref{6:23:1}, therefore follows from Theorem \ref{thm:ar:1:3a} that any
	full trajectory $\gamma=\{(u(t),u_t(t),\theta(t)):t\in\mathds{R}\}$ belongs to $\mathfrak{A}_\nu$
	enjoys the following regularity properties
	\begin{eqnarray}
		u_t\in L^\infty(\mathbb{R},V_1)\cap C(\mathbb{R},V_0),\ \ u_{tt}\in L^\infty(\mathbb{R},V_0)\quad\text{and} \quad \theta_t\in L^\infty(\mathbb{R},V_1).\label{8:2}
	\end{eqnarray}
	From $\eqref{p:1:1}_2$ and \eqref{8:2}, we have
	\begin{eqnarray}
		A\theta=-\theta_t-\delta A^\sigma u_t\in L^\infty (\mathds{R},V_0),\label{8:3}
	\end{eqnarray}
	therefore $\theta\in L^\infty(\mathds{R};V_2)$. From $\eqref{p:1:1}_1$, \eqref{8:2}  and the fact that nonlinear term $f$ is locally Lipschitz, we obtain
	\begin{eqnarray}
		Au=h-u_{tt}-A^\nu u_t+\delta A^\sigma\theta -f(u)\in L^\infty(\mathds{R};V_0),
	\end{eqnarray}
	consequently $u\in L^\infty(\mathds{R},V_2)$. Thus, \eqref{8:1} holds.
\end{proof}

\begin{theorem}
	The dynamical system $(\pazocal{H},S_\beta(t))$ possesses a generalized exponential attractor $\mathfrak{A}_\nu^{exp}\subset\pazocal{H}$, with finite fractal dimension in extended space
	\begin{eqnarray}
		\pazocal{H}_{-1}=H^{-1}(\Omega)\times L^2(\Omega)\times H^{-1}(\Omega)\supset \pazocal{H}.
	\end{eqnarray}
\end{theorem}

\begin{proof}
	Consider the bounded positively invariant absorbing set $\mathcal{B}$ given in Remark \ref{rem:4:5}, therefore $(\pazocal{H},S_\beta(t))$ is quasi-stable on $\mathcal{B}$. Moreover, given $U_0\in\mathcal{B}$ such that $U(t)=S_\beta(t)U_0$ is a strong solution of \eqref{4:4}, then there exists a positive constant $C_{\mathcal{B}}$ such that for any $T$, we have
	\begin{eqnarray}
		\|S_\beta(t)U_0\|_{\pazocal{H}}\leq C_{\mathcal{B}},\quad \forall t\in[0,T].
	\end{eqnarray}
	Since $\pazocal{F}$ is locally Lipschitz  and $\mathcal{B}$ is positively invariant, from \eqref{4:4} we obtain
	\begin{eqnarray}
		\|U_t(t)\|_{\pazocal{H}_{-1}}\leq \|\pazocal{A}U(t)\|_{\pazocal{H}}+\|\pazocal{F}(U(t))\|_{\pazocal{H}}\leq C_{\mathcal{B}T},\quad \forall t\in[0,T],
	\end{eqnarray}
	where $C_{\mathcal{B}T}>0$ is a constant independent of $\nu\in(0,1/2)$. Then we have
	\begin{eqnarray}
		\|S_\beta(t_1)U_0-S_\beta(t_2)U_0\|_{\pazocal{H}_{-1}}\leq \int_{t_1}^{t_2}\|U_t(s)\|_{\pazocal{H}_{-1}}ds\leq C_{\mathcal{B}T}|t_1-t_2|,\quad 0\leq t_1<t_2\leq T.\label{6:3}
	\end{eqnarray}
	This means that $t\mapsto S_\beta(t)U_0$ is Hölder continuous in extended space $\pazocal{H}_{-1}$. By using a density argument, \eqref{6:3} is valid for mild solutions. Therefore, the existence of exponential attractor (with finite fractal dimensional in $\pazocal{H}_{-1}$) follows from Theorem \ref{thm:ar:1:4}.
\end{proof}

\section{Upper semicontinuity}
In Theorem \ref{thm:4:1} we prove that for any $\sigma,\nu\in[0,1]$ there exists a compact global attractor $\mathfrak{A}_\beta$, $\beta=(\nu,\sigma)$. In this section, we prove that the family $\{\mathfrak{A}_\beta\}_{\beta\in\Lambda}$, $\Lambda=[0,1]\times [0,1]$ is upper-semicontinuous at $\beta_0=(0,0)$.

\begin{theorem}
	The family of global attractors $\{\mathfrak{A}_\beta\}_{\beta\in\Lambda}$ is upper-semicontinuous at $\beta_0=(0,0)$, that is 
	\begin{eqnarray}
		\lim_{\beta\to 0}d_{\pazocal{H}}\{\mathfrak{A}_\beta|\mathfrak{A}_{\beta_0}\}=0.\label{9:1}
	\end{eqnarray}
\end{theorem}

\begin{proof}
	Let $U_0^n\in\mathfrak{A}_{\beta_n}$ where $\beta_n=(\nu_n,\sigma_n)$, with $\nu_n\to 0^+$ and $\sigma_n\to 0^+$. Let us to show that there exists a subsequence, still represented by $(U^n_0)$, such that $U_0^n\to U_0\in\mathfrak{A}_{\beta_0}$ in $\pazocal{H}$. Since $U^n_0\in\mathfrak{A}_{\beta_n}$ there exists a full tragectory $U^n(t)=(u^n(t),u_t^n(t),\theta(t))$ such that $U^n(0)=U^n_0$. Consider the bounded positively invariant absorbing set $\mathcal{B}$ defined in Remark \ref{rem:4:5}. Therefore,
	\begin{eqnarray}
		\|U^n(t)\|^2_{\pazocal{H}}=\|(u^n(t),u^n_t(t),\theta(t))\|^2_{\pazocal{H}}\leq R_0^2,\quad\forall t,n.\label{9:2}
	\end{eqnarray}
	From Theorem \ref{thm:5:1a} and Theorem \ref{thm:ar:1:3a} there exist $R_1,R_2>0$, such that 
	\begin{eqnarray}
		\|u^n_t(t)\|^2_{V_1}+\|u^n_{tt}(t)\|^2_2+\|\theta^n_t(t)\|^2_{V_1}\leq R_1^2,\quad \forall t,n\label{9:3}
	\end{eqnarray} 
	and
	\begin{eqnarray}
		\|u^n(t)\|^2_{V_2}+\|\theta^n(t)\|^2_{V_2}\leq R_2^2,\quad \forall t,n.\label{9:4}
	\end{eqnarray} 
	For every compact interval $I\subset\mathds{R}$, by using \eqref{9:3}, \eqref{9:4} and the Arzelá-Ascoli Theorem, we obtain $U\in C(I;\pazocal{H})$ such that $U^{n_k}\to U$ in $C(I;\pazocal{H})$ for some subsequence $(U^{n_k})$ of $(U^n)$. Being $I$ arbitrary, by using a diagonal argument we can obtain $U\in C(\mathds{R};\pazocal{H})$ such that, for any compact interval $I\subset\mathds{R}$, we have 
	\begin{eqnarray}
		U^{n_k}\to U=(u,u_t,\theta)\quad\text{in}\quad C(I;\pazocal{H}),\label{9:5}
	\end{eqnarray}
	and from \eqref{9:2}, we obtain
	\begin{eqnarray}
		\sup_{t\in\mathds{R}}\|U(t)\|_{\pazocal{H}}\leq R_0.
	\end{eqnarray}
	Now, let us to prove that $U=(u,u_t,\theta)$ is weak solution to
	\begin{eqnarray}
		\begin{aligned}
			u_{tt}-A u+u_t+\delta \theta
			+f(u)=&\ h,\quad\text{in} \ \Omega\times(0,\infty),\\
			\theta_t- A\theta+\delta u_{t}=&\ 0,\quad\, \text{in}\ \Omega\times(0,\infty).\label{9:6}
		\end{aligned}
	\end{eqnarray}
	Note that, from definition of weak solution to \eqref{p:1:1}-\eqref{p:1:3}, the functions $U^{n_k}=(u^{n_k},u^{n_k}_t,\theta^{n_k})$ satisfies the following identity (in the sense of distributions)
	\begin{eqnarray}
		\begin{aligned}
			\frac{d}{dt}(u^{n_k}_t,\varphi)_2+(A^{1/2}u^{n_k},A^{1/2}\varphi)_2+(A^{\nu_k} u^{n_k}_t,\varphi
			)_2-\delta(A^{\sigma_k}\theta^{n_k},\varphi)_2+\\
			(f(u^{n_k}),\varphi)_2=(h,\varphi)_2,\\
			\frac{d}{dt}(\theta^{n_k},\phi)_2+(A^{1/2}\theta^{n_k},A^{1/2}\phi)_2+\delta (A^{\sigma_k} u^{n_k}_t,\phi)_2=0,
		\end{aligned}\label{9:7}
	\end{eqnarray}
	for all $\varphi,\phi\in V_1$. From \eqref{9:5}, we have
	\begin{eqnarray}
		\begin{aligned}
			\lim_{k\to\infty}(u^{n_k}_t,\varphi)_2=(u_t,\varphi)_2,\quad \lim_{k\to\infty}(A^{1/2}u^{n_k},A^{1/2}\varphi)_2=(A^{1/2}u,A^{1/2}\varphi)_2\quad\text{and}\\
			\lim_{k\to\infty}(\theta^{n_k},\phi)_2=(\theta,\phi)_2.
		\end{aligned}\label{9:8}
	\end{eqnarray}
	Note that \eqref{9:4} implies 
	\begin{eqnarray}
		\theta^{n_k}\to\theta, \quad\text{weakly in}\quad L^\infty (I;H_0^1(\Omega)).
	\end{eqnarray}
	Then
	\begin{eqnarray}
		\lim_{k\to\infty}(A^{1/2}\theta^{n_k},A^{1/2}\phi)_2=(A^{1/2}\theta,A^{1/2}\phi)_2.\label{9:9}
	\end{eqnarray}
	From \eqref{9:5}, we have
	\begin{eqnarray}
		u^{n_k}\to u\quad\text{strongly in}\quad L^2(I;V_0).
	\end{eqnarray}
	Hence, by \eqref{3:2}, we get
	\begin{eqnarray}
		\lim_{k\to\infty}(f(u^{n_k}),\varphi)_2=(f(u),\varphi)_2.\label{9:10}
	\end{eqnarray}
	By using an argument analogous to found in \cite{doi:10.1007/s00245-019-09590-1,doi:10.1016/j.jde.2021.01.030}, we can prove that
	\begin{eqnarray}
		\begin{aligned}
			\lim_{k\to\infty}(A^{\nu_k}u^{n_k}_t,\varphi)_2=(u_t,\varphi)_2,\quad 
			\lim_{k\to\infty}(A^{\sigma_k}\theta^{n_k},\varphi)_2=(\theta,\varphi)_2,\quad\text{and}\\ 
			\lim_{k\to\infty}(A^{\sigma_k}u_t^{n_k},\phi)_2=(u_t,\phi)_2.
		\end{aligned}\label{9:11}
	\end{eqnarray}
	By using the convergences \eqref{9:8}, \eqref{9:9}, \eqref{9:10} and \eqref{9:11}, we can pass to the limit in all terms of \eqref{9:7}. Therefore, $U$ is a full bounded trajectory of \eqref{9:6} and thus $U(0)\in\mathfrak{A}_{\beta_0}$. Finally, by \eqref{9:5} we obtain
	\begin{eqnarray}
		U^{n_k}_0=U^{n_k}(0)\to U(0)\quad\text{in}\quad\pazocal{H},
	\end{eqnarray}
	which concludes the proof of the theorem.
\end{proof}

\appendix

\section{Known results for infinite-dimensional dynamical systems}\label{sec:4}

This section intends to present some definitions and results related to the Theory of Infinite-Dimensional Dynamical Systems that will be used to obtain the main results of this work. Such subjects can be found in \cite{isbn:9780080875460,isbn:9789667021641,isbn:9780821815274,isbn:9781461206453} and \cite{doi:10.1007/978-3-319-22903-4,isbn:9780821866535,ISBN:9780387-877112} (for notion of quasi-stability).
In this section, $T_t$, $t\geq 0$, represents a strongly continuous nonlinear semigroup on a Banach space $(\mathcal{H},|\cdot|_{\mathcal{H}})$. Therefore, the symbol $(\mathcal{H},T_t)$ denotes a dynamical system on the phase space $\mathcal{H}$.

\begin{definition}\label{def:AR:1}
	We say that $(\mathcal{H},T_t)$ is {\it dissipative}, if there exists a bounded {\it absorbing}  set $\mathcal{B}\subset\mathcal{H}$ for $(\mathcal{H},T_t)$, that is, for every bounded set $\mathcal{D}\subset\mathcal{H}$, there is $T_{\mathcal{D}}>0$, such that
	\begin{eqnarray}
		T_t\mathcal{D}\subset\mathcal{B},\quad\text{for all}\quad t\geq T_{\mathcal{D}}.
	\end{eqnarray}
	In this case, we call \textit{radius of dissipativity} at the value $R>0$, such that $\mathcal{B}\subset\{x\in H; \ \|x\|_{H}\leq R\}$ where $\mathcal{B}$ is an absorbing bounded set for $(\mathcal{H},T_t)$.
\end{definition}

\begin{definition}\label{def:AR:2}
	A \textit{global attractor} for $(\mathcal{H},T_t)$ is a closed bounded set ${\bf A}\subset\mathcal{H}$ that enjoy following properties:
	\begin{description}
		\item[a)] ${\bf A}$ is invariant with respect to the $T_t$, that is
		$$T_t{\bf A}={\bf A},\quad\forall t\geq0;$$
		\item[b)] {\bf A} uniformly attracts bounded sets, that is, for every bounded set $\mathcal{D}\subset\mathcal{H}we say that$, we have
		\begin{eqnarray}
			\lim_{t\to\infty}d_{\mathcal{H}}\{T_t\mathcal{D}|{\bf A}\}=0,
		\end{eqnarray}
		where $d_{\mathcal{H}}\{A|B\}$ is \textit{Hausdorff's semi-distance} between sets $A$ and $B$, that is	
		\begin{eqnarray}
			d_{\mathcal{H}}\{A|B\}=\sup_{x\in A}\text{dist}(x,B),\quad\text{with}\quad \mbox{dist}(x,B)=\inf_{y\in B}|x-y|_{\mathcal{H}}.
		\end{eqnarray}
	\end{description}
\end{definition}

\begin{remark}\label{rem:ra:1}
	Note that if $(\mathcal{H},T_t)$ possesses a global attractor ${\bf A}$, then, by Definition \ref{def:AR:2} {\bf (b)}, $(\mathcal{H},T_t)$ is dissipative, since any $\varepsilon$-neighborhood of ${\bf A}$ is an absorbing set for $(\mathcal{H},T_t)$. 
\end{remark}

\begin{definition}\label{def:AR:3}
	We say that $(\mathcal{H},T_t)$ is \textit{asymptotically smooth} if for any positively invariant bounded set $\mathcal{D}\subset\mathcal{H}$,  that is, $T_t\mathcal{D}\subset\mathcal{D}$ for all $t>0$, there
	exists a compact set $\mathcal{K}\subset\overline{\mathcal{D}}$,
	where  $\overline{\mathcal{D}}$ is the closure of $\mathcal{D}$,
	such that
	\begin{eqnarray}
		\lim_{t\to
			+\infty}d_{\mathcal{H}}\{T_t\mathcal{D}|\mathcal{K}\}=0.
	\end{eqnarray}
\end{definition}

A set of relevant importance for us is that of the {\it stationary points} of $(\mathcal{H},T_t)$, denoted by $\mathcal{N}$, that is,
\begin{eqnarray}
	\mathcal{N}=\{w \in\mathcal{H}; \ T_tw=w, \ \forall t>
	0\}.\nonumber
\end{eqnarray}

\begin{definition}\label{def:AR:4}
	The set of all $w\in\mathcal{H}$ such that there is a full trajectory $\gamma=\{u(t); \ t
	\in\mathbb{R}\} $ satisfying
	\begin{eqnarray}
		u(0)=w\quad\text{and}\quad \lim_{t\to-\infty}\text{dist}
		(u(t),\mathcal{N})=0.
	\end{eqnarray}
	is called the \textit{unstable manifold} emanating from $\mathcal{N}$ and is represented by $\mathcal{M}^u(\mathcal{N})$. Note that $\mathcal{M}^u(\mathcal{N})$ is an invariant set
	for $(\mathcal{H},T_t)$ and if ${\bf A}\subset\mathcal{H}$ is global
	attractor for $(\mathcal{H},T_t)$, then
	$\mathcal{M}^u(\mathcal{N})\subset{\bf A}$ (cf. \cite{isbn:9780080875460,ISBN:9780387-877112}).
\end{definition}

\begin{definition}\label{def:AR:5}
	We say that $(\mathcal{H},T_t)$ is \textit{gradient}, if there exists a strict Lyapunov function on
	$\mathcal{H}$, that is, there exists a continuous function $\Phi:\mathcal{H}\to\mathds{R}$ such that, $t\mapsto\Phi(T_tw)$ is non-increasing for any $w\in\mathcal{H}$, and if $\Phi(T_tw)=\Phi(w)$ for all $t>0$, then $w$ is a stationary point of
	$(\mathcal{H},T_t)$.
\end{definition}

\begin{definition}\label{def:AR:6}
	Let $\mathcal{K}\subset\mathcal{H}$ be a compact set, the fractal dimension of $\mathcal{K}$, represented by $\text{dim}_f^{\mathcal{H}}\mathcal{K}$, is the value given by limit
	\begin{eqnarray}
		\text{dim}_f^{\mathcal{H}}\mathcal{K}=\lim_{\varepsilon\to 0}\sup\frac{\ln
			n(M,\varepsilon)}{\ln(1/\varepsilon)},\nonumber
	\end{eqnarray}
	where $n(\mathcal{K},\varepsilon)$ is the minimal number of closed balls of
	radius $\varepsilon$ which covers $\mathcal{K}$.
\end{definition}

\begin{theorem}[\cite{ISBN:9780387-877112}]\label{thm:ar:1:1}
	If $(\mathcal{H},T_t)$ is a asymptotically smooth dynamical system satisfying:
	\begin{description}
		\item[a)] it is gradient with Lyapunov function $\Phi$ bounded from above on any bounded subset of $\mathcal{H}$;
		\item[b)] for every $R>0$, the set $\Phi_R:=\{w;\ \Phi(w)\leq R\}$ is bounded;
		\item[c)] the stationary points set $\mathcal{N}$ is bounded.
	\end{description} 
	Then $(\mathcal{H},T_t)$ possesses a compact global attractor ${\bf A}=\mathcal{M}^u(\mathcal{N})$.
\end{theorem}

\begin{assumption}\label{ar:assump:1}
	Let $(\mathcal{H},T_t)$ be a dynamical system such that: 
	\begin{description}
		\item[a)] the phase space is of the form $\mathcal{H}=X\times Y\times Z$, where $(X,|\cdot|_X)$, $(Y,|\cdot|_Y)$ and $(Z,|\cdot|_Z)$ are reflexives Banach spaces with $X$ compactly embedded in $Y$ and  $|\cdot|^2_{\mathcal{H}}=|\cdot|^2_X+|\cdot|^2_Y+|\cdot|^2_Z$;
		\item[b)] the  evolution operator $T_t$ is of the form
		\begin{eqnarray}
			T_tw_0=(\phi(t),\phi_t(t),\vartheta(t)),\quad t\geq 0,\quad
			w_0=(\phi(0),\phi_t(0),\vartheta(0))\in\mathcal{H},\label{ar:1:2}
		\end{eqnarray}
		where the functions $\phi$ and $\vartheta$ have regularities
		\begin{eqnarray}
			\phi\in C(\mathbb{R}_+,X)\cap C^1(\mathbb{R}_+,Y)\quad\text{and}\quad \vartheta\in C(\mathbb{R}_+,Z).\label{ar:1:3}
		\end{eqnarray}
	\end{description}
	
\end{assumption}

\begin{definition}\label{ar:def:1:1}
	We say that a dynamical system $(\mathcal{H},T_t)$ satisfying Assumption \ref{ar:assump:1}, is {\it quasi-stable} on $\mathcal{B}\subset\mathcal{H}$, if there exist a compact seminorm $\eta_X$ on $X$ and nonnegative funtions $a(t)$, $b(t)$ and $c(t)$ defined in $\mathds{R}_+$ such that
	\begin{description}
		\item[a)] $a(t)$ and $c(t)$ are locally bounded;
		\item[b)] $b(t)\in L^1(\mathds{R}_+)$ and $\lim_{t\to\infty}b(t)=0$;
		\item[c)] for every $w_1,w_2\in\mathcal{B}$ and $t>0$, if $T_tw_i=(\phi^i(t),\phi_t^i(t),\vartheta^i(t))$, $i=1,2$, then  the following relations
		\begin{eqnarray}
			\|T_tw_1-T_tw_2\|_{\mathcal{H}}^2\leq
			a(t)\|w_1-w_2\|_{\mathcal{H}}^2\label{3:1:5}
		\end{eqnarray}
		and
		\begin{eqnarray}
			\|T_tw_1-T_tw_2\|_{\mathcal{H}}^2\leq b(t)\|w_1-w_2\|_{\mathcal{H}}^2+c(t)\sup_{0\leq s\leq
				t}[\eta_X(\phi^1(s)-\phi^2(s))]^2\label{3:1:6}
		\end{eqnarray}
		hold.
	\end{description}
\end{definition}

The following results can be found in \cite[Chapter
7]{ISBN:9780387-877112}, show us how strong the property of
quasi-stability is for a dynamical system. The first relates the
quasi-stability to the asymptotically smooth and the second relates
the quasi-stability to the fractal dimension of an attractor.

\begin{theorem}\label{thm:ar:1:2}
	Let $(\mathcal{H},T_t)$ be a dynamical system satisfying Assumption \ref{ar:assump:1}. If $(\mathcal{H},T_t)$ is quasi-stable on every bounded positively  invariant set $\mathcal{B}\subset\mathcal{H}$, then $(\mathcal{H},T_t)$ is asymptotically smooth.
\end{theorem}

\begin{theorem}\label{thm:ar:1:3}
	Let $(\mathcal{H},T_t)$ be a dynamical system satisfying Assumption \ref{ar:assump:1}. If $(\mathcal{H},T_t)$ possesses a compact global attractor ${\bf A}$ and is quasi-stable on ${\bf A}$. Then the fractal
	dimension of ${\bf A}$ is finite.
\end{theorem}

\begin{theorem}\label{thm:ar:1:3a}
	Let $(\mathcal{H},T_t)$ be a dynamical system satisfying Assumption \ref{ar:assump:1} that possesses a global attractor ${\bf A}$. If $(\mathcal{H},T_t)$ is quasi-stable on ${\bf A}$ and $ c_\infty=\sup_{t \in\mathds{R}_+} c(t)<\infty$ in \eqref{3:1:6}. Then any full trajectory $\gamma=\{(\phi(t),\phi_t(t), \vartheta(t)): t \in\mathds{R}\} $ belongs to the global attractor enjoys the following regularity properties,
	\begin{eqnarray}
		\phi_t\in L^\infty(\mathds{R};X)\cap C(\mathds{R};Y), \quad \phi_{tt}\in L^\infty(\mathds{R};Y),\quad \vartheta_t \in L^\infty(\mathds{R}; Z).
	\end{eqnarray}
	In addition, there exist $R>0$, depending on $c_\infty $, on seminorm $\eta_X$ and on embedding of $X$  into $Y$ such that
	\begin{eqnarray}
		|\phi_t(t)|_X^2 + |\phi_{tt}(t)|_Y^2+|\vartheta_t(t)|_Z^2\leq R, \quad t\in\mathds{R}.
	\end{eqnarray}
\end{theorem}

\begin{definition}\label{def:AR:7}
	We say that a compact set ${\bf A}_{\exp}\subset\mathcal{H}$ is a {\it fractal exponential attractor} of $(\mathcal{H},T_t)$ if it has finite fractal dimension, is positively invariant, and for any bounded set $\mathcal{B}\subset\mathcal{H}$, there
	exist
	constants $T_{\mathcal{B}}, C_{\mathcal{B}}>0$ and $\gamma_{\mathcal{B}}>0$ such that for all
	$t\geq T_{\mathcal{B}}$,$$
	\mbox{dist}(T_t\mathcal{B},{\bf A}_{\exp})\leq
	C_{\mathcal{B}}\exp(-\gamma_{\mathcal{B}}(t-T_{\mathcal{B}})). $$
	In some cases, one can prove the existence of an exponential attractor whose dimension is finite in some extended space $\widetilde{\mathcal{H}}\supset \mathcal{H}$ only.
\end{definition}

\begin{theorem}\label{thm:ar:1:4}
	Let $(\mathcal{H},T_t)$ be dissipative, satisfying the Assumption \ref{ar:assump:1} and quasi-stable on some bounded absorbing set $\mathcal{B}$. If there exists a space $\widetilde{\mathcal{H}} \supset\mathcal{H}$ such that, for every $w\in\mathcal{B}$, $t\mapsto T_tw$ is Hölder continuous in $\widetilde{\mathcal{H}}$, that is, there exist $0<\gamma\leq 1$ and $C_{\mathcal{B},T}>0$ such that
	\begin{eqnarray}
		\|T_{t_1}w-T_{t_2}w\|_{\widetilde{\mathcal{H}}}\leq C_{\mathcal{B},T}|t_1-t_2|^\gamma,\quad \forall t_1,t_2\in[0,T],\ \forall w\in\mathcal{B}.
	\end{eqnarray}
	Then $(\mathcal{H},T_t)$ possesses a generalized fractal exponential attractor with finite dimension in $\widetilde{\mathcal{H}}$.
\end{theorem}

\begin{definition}\label{def:A:9}
	Consider a family of dynamical systems (on the same phase space $\mathcal{H}$) $\{(\mathcal{H},T^\lambda_t)),\lambda\in\Lambda\}$, indexed in a metric space $\Lambda$, such that, for every $\lambda\in\Lambda$, $(\mathcal{H},T_t^\lambda)$ possesses a global attractor represented by ${\bf A}_{\lambda}$. We say that the family of attractors $\{{\bf A}_{\lambda}\}_{\lambda\in\Lambda}$ is upper semicontinuous at $\lambda_0\in\Lambda$, if
	\begin{eqnarray}
		\lim_{\lambda\in\Lambda\to\lambda_0}d_{\mathcal{H}}\{{\bf A}_{\lambda}|{\bf A}_{\lambda_0}\}=0.\label{A:11}
	\end{eqnarray}
\end{definition}

{\bf Acknowledgements}
\\
The authors are grateful to the referees for the valuable suggestions that improved this paper.


\EditInfo{August 21, 2023}{December 5, 2023}{Serena Dipierro}

\end{document}